\documentclass[10pt,a4paper]{article}

\usepackage[comma]{natbib}
\usepackage{color}

\usepackage{amsmath,amssymb,amsthm,amsopn}
\usepackage{multirow}
\usepackage{float}
\usepackage{subfig}
\usepackage{graphicx}

\allowdisplaybreaks

\newtheorem{lemma}{Lemma}
\newtheorem{theorem}{Theorem}
\theoremstyle{remark}
\newtheorem{remark}{Remark}

\DeclareMathOperator{\Var}{Var}

\usepackage{verbatim,graphicx}

\RequirePackage[latin1]{inputenc} \RequirePackage[english]{babel}
\RequirePackage{amssymb} \RequirePackage{amsfonts}
\RequirePackage{amsmath} \RequirePackage{amsthm}
\RequirePackage{latexsym}

\newcommand{\abs}[1]{\left\vert#1\right\vert}

\theoremstyle{plain}

\newtheorem{teor*}{Teorema}

\theoremstyle{definition}

\makeatletter
\def\cleardoublepage{\clearpage\if@twoside \ifodd\c@page\else
    \hbox{}
    \vspace*{\fill}
    \vspace{\fill}
    \thispagestyle{empty}
    \newpage
    \if@twocolumn\hbox{}\newpage\fi\fi\fi}
\makeatother

\title{Fourier methods for smooth distribution function estimation}

\author{J.E. Chac\'on\footnote{Departamento de Matem\'aticas, Universidad de Extremadura, E-06006 Badajoz, Spain. Email: {\tt jechacon@unex.es}}, \  P. Monfort\footnote{Departamento de Matem\'aticas, Universidad de Extremadura, E-06006 Badajoz, Spain. Email: {\tt pabmonf@unex.es}} \  and  C. Tenreiro\footnote{CMUC, Department
of Mathematics, University of Coimbra, Apartado 3008, 3001-454 Coimbra, Portugal. E-mail: {\tt
tenreiro@mat.uc.pt}}}

\begin{document}

\maketitle

\begin{abstract}
\noindent The limit behavior of the optimal bandwidth sequence for the kernel distribution function
estimator is analyzed, in its greatest generality, by using Fourier transform methods. We show a
class of distributions for which the kernel estimator achieves a first-order improvement in
efficiency over the empirical estimator.
\end{abstract}


\newpage

\section{Introduction}\label{sec:intro}

The kernel estimator of a distribution function was introduced independently by \citet{Tia63},
\citet*{Nad64} and \citet{WL64} as a smooth alternative to the empirical estimator. It is defined
as the distribution function corresponding to the well-known kernel density estimator. Precisely,
given independent real random variables $X_1, \ldots , X_n$  with common and unknown distribution
function $F$, assumed to be absolutely continuous with density function $f$, the kernel estimator
of $F(x)$ is $$F_{nh}(x)=n^{-1}\sum_{j=1}^nK\big(h^{-1}(x-X_j)\big),$$ where $h>0$ is the bandwidth
and the function $K$ will be referred to as the integrated kernel, since it is assumed that
$K(x)=\int_{-\infty}^xk(y)dy$ for some integrable function $k$, called kernel, having unit integral
over the whole real line.

Classical references on kernel distribution function estimators include \cite{Yam73}, which
provided mild necessary and sufficient conditions for its consistency in uniform norm,
\cite{Azz81}, \cite{Swa88} and \cite{Jon90} on asymptotic squared error analysis of the estimator,
or \cite{Sar93}, \cite{Alt95} and \cite{BHP98}, and more recently \cite{PB00} and \cite{Ten06}, on
data-driven bandwidth selection. There are also other recent papers on different aspects of kernel
distribution function estimation, like \cite{Ten03}, \cite{SvG05}, \cite{JSV07}, \cite{GN09},
\cite{BP09}, \cite{CRC10}, \cite{MS12} or \cite{Ten13}. See \cite{Ser09} for a detailed survey on
distribution function estimation, not limited to kernel-type methods.

This paper is devoted to the study of the kernel distribution function estimator from the point of
view of the mean integrated squared error,  $${\rm MISE}(h)\equiv{\rm MISE}_n(h)=\mathbb
E\int_{-\infty}^\infty\{F_{nh}(x)-F(x)\}^2dx.$$ In this sense, the optimal bandwidth  $h_{0n}$ is
the value of $h>0$ minimizing ${\rm MISE}(h)$. The existence of such a bandwidth was proved in
Theorem 1 of \cite{Ten06} under very general assumptions, and Proposition 2 in the same paper
showed that $h_{0n}\to0$ whenever the Fourier transform of $k$ is not identically equal to 1 on any
neighbourhood of the origin. This condition can be considered mild as well, since it is satisfied
for any finite-order kernel; however, it does not hold for a superkernel \citep[see][]{CMN07}.

The purpose of this note is to show how to use Fourier transform techniques for the analysis of kernel distribution estimators. Particularly, expressing the MISE in terms of characteristic functions allows us to obtain a result on the limit behavior of the optimal bandwidth
sequence in its most general form so that it also covers the case of a superkernel, and to explore
its consequences showing the peculiar properties of the use of superkernels and the sinc kernel in
kernel distribution function estimation. Precisely, it is shown in Section
\ref{sec:mainresults} that in some situations the sequence $h_{0n}$ does not necessarily tend to
zero. Moreover, we exhibit a class of distributions for which the kernel distribution estimator
presents a first-order improvement over its empirical counterpart, opposite to the usual situation, where only second-order improvements are possible (see Remark \ref{rem3}). Our findings are illustrated in Section \ref{sec:examples} through two representative examples.

\section{Main results}\label{sec:mainresults}

Recall from \cite{CRC10} that the kernel distribution function estimator admits the representation
\begin{equation}\label{repr}
F_{nh}(x)=\int F_n(x-hz)dK(z),
\end{equation}
where $F_n$ denotes the empirical distribution function (here and below integrals without integration limits are meant over the whole real line). Using this, and standard properties of the empirical process, it is possible to obtain a decomposition of ${\rm MISE}(h)={\rm IV}(h)+{\rm ISB} (h)$, where the integrated variance ${\rm IV}(h)=\int\Var\{F_{nh}(x)\}dx$ and the integrated squared bias ${\rm ISB}(h)=\int\{\mathbb E[F_{nh}(x)]-F(x)\}^2dx$ can be expressed in the following exact form:
\begin{align}
{\rm IV}(h)&=n^{-1}\iiint\big\{F\big(x-h(y\vee z)\big)-F(x-hy)F(x-hz)\big\}dK(y)dK(z)dx,\label{IV}\\
{\rm ISB}(h)&=\iiint\{F(x-hy)-F(x)\}\{F(x-hz)-F(x)\}dK(y)dK(z)dx,\label{ISB}
\end{align}
with $y\vee z$ standing for $\max\{y,z\}$.

Note that the representation (\ref{repr}) and the exact expressions (\ref{IV}) and (\ref{ISB}) also
make sense for $h=0$, implying that the kernel distribution estimator reduces to the empirical
distribution function for $h=0$, for which the well-known MISE formula reads ${\rm MISE}(0)={\rm
IV}(0)=n^{-1}\int F(1-F)$ whenever $\psi(F)=\int F(1-F)$ is finite. Moreover, it is not hard to
check that $\int|x|dF(x)<\infty$ and $\int|y\,k(y)|dy<\infty$ ensure that MISE$(h)$ is finite for
all $h>0$, so those two minimal conditions will be assumed henceforth.  Note that the required
condition that $F$ have a finite mean is slightly stronger than $\psi(F)<\infty$ since
$\psi(F)\leq2\int|x|dF(x)$.

\subsection{Limit behavior of the optimal bandwidth sequence}

Denote by $\varphi_g$ the Fourier transform of a function $g$, defined as $\varphi_g(t)=\int e^{itx}g(x)dx$. As in \cite{CMNP07}, the key to understand the limit behavior of the optimal bandwidth sequence is
to use Fourier transforms to express the MISE criterion. \cite{Abd93} provided a careful account of the
necessary conditions under which the MISE can be expressed in terms of Fourier transforms. The proof of his Proposition 2 implicitly derives formulas for ${\rm ISB}(h)$ and ${\rm IV}(h)$ in terms of $\varphi_k$ and $\varphi_f$ for $h>0$. We reproduce this result here for completeness, and show that it can be extended to cover the case $h=0$ as well. 


\begin{theorem}\label{thm:Fourier}
If $\int|x|dF(x)<\infty$ and $\int|y\ k(y)|dy<\infty$ then, for all $h\geq0$, the IV and ISB functions can be written
as
\begin{align*}
{\rm IV}(h)&=(2\pi)^{-1}n^{-1}\int t^{-2}|\varphi_k(th)|^2\{1-|\varphi_f(t)|^2\}dt,\\
{\rm ISB}(h)&=(2\pi)^{-1}\int t^{-2}|1-\varphi_k(th)|^2|\varphi_f(t)|^2dt.
\end{align*}
\end{theorem}

Particularly, note that for $h=0$ the previous result yields a Parseval-like formula for distribution functions,
\begin{equation}\label{eq:F1F}
\psi(F)=\int F(1-F)=(2\pi)^{-1}\int t^{-2}\{1-|\varphi_f(t)|^2\}dt,
\end{equation}
which can be useful to compute errors in an exact way in cases where $F$ does not have a close
expression but $\varphi_f$ does, as it happens for instance for the normal distribution (see also
Section \ref{sec:examples} below). Moreover, we show in Lemma \ref{lem:intF1F} that (\ref{eq:F1F})
remains valid for integrated kernels $K$. In the following it will be assumed that $\psi(K)>0$, a
property that immediately holds, using (\ref{eq:F1F}), whenever $\varphi_k(t)\in[0,1]$ for all $t$.
Note that, for density estimation, admissible kernels are precisely those whose Fourier transform
satisfies that restriction \citep[see][]{Cli88}.

The limit behavior of the optimal bandwidth sequence $h_{0n}$ is determined in its greatest
generality by the following constants, depending on the Fourier transforms of $f$ and $k$: let
$C_f$ denote the smallest positive frequency from which $\varphi_f$ is null along a proper interval
and $D_f$ the positive frequency from which $\varphi_f$ is identically null (so that $C_f\leq D_f$,
both possibly being infinite); also, denote $S_k$ the greatest frequency such that $\varphi_k$ is
identically equal to one on $[0,S_k]$ and $T_k$ the smallest frequency such that $\varphi_k$ is not
identically equal to one on a subinterval of $[T_k,\infty)$, and note that $S_k\leq T_k$ with both
possibly being zero. In mathematical terms,
$$\begin{array}{ll}
C_f=\sup\{r\ge 0\colon\varphi_f(t)\neq0\text{ a.e. for }t\in[0,r]\},&\quad
D_f=\sup\{t\ge0\colon\varphi_f(t)\neq0\}\\
S_k=\inf\{t\ge 0\colon\varphi_k(t)\neq1\},
&\quad T_k=\inf\{r\ge0\colon\varphi_k(t)\neq1\text{ a.e. for }t\ge r\}.
\end{array}$$
Finally, define $h_*=\sup\{h\geq0\colon{\rm ISB}(h)=0\}$. The following result shows the limit of the optimal bandwidth sequence $h_{0n}$ in the common
case where $C_f=D_f$ and $S_k=T_k$.

\begin{theorem}\label{thm:limit}
Assume that $\int|x|dF(x)<\infty$, $\int|y\, k(y)|dy<\infty$ and $\psi(K)>0$, and suppose that
$C_f=D_f$ and $S_k=T_k$. Then, $h_{0n}\to S_k/D_f$ as $n\to\infty$ and also $h_*=S_k/D_f$.
\end{theorem}

A number of consequences can be extracted from Theorem \ref{thm:limit}:

\begin{remark}
A kernel $k$ with $S_k>0$ is called a superkernel \citep[see][]{CMN07}. If an integrated
superkernel is used in the kernel distribution function estimator and the density $f$ is such that
$D_f<\infty$ (see \citealp{CMNP07}, and Section \ref{sec:examples} below for examples of such
distributions) then, contrary to the usual situation, the optimal bandwidth sequence $h_{0n}$ does
not tend to zero, but to the strictly positive constant $S_k/D_f$. Moreover, any positive constant can be the limit of an optimal bandwidth sequence, because modifying the scale of the density by taking $f_a(x)=f(x/a)/a$, for any $a>0$, it follows that $D_{f_a}=D_f/a$, and hence the limit of the optimal bandwidth sequence equals $aS_k/D_f$.
\end{remark}

\begin{remark}
Since $h_*=S_k/D_f$, the kernel estimator $F_{nh}$ is unbiased for any fixed (i.e., not depending
on $n$)
    choice of $h\in[0,S_k/D_f]$. If either $K$ is not an integrated superkernel or the characteristic function $\varphi_f$ does not have bounded support, then the only kernel distribution estimator with null ISB corresponds to $h=0$, the empirical distribution function.
\end{remark}
\begin{remark}\label{rem3}
It is shown in the proof of Theorem \ref{thm:limit} that for $h\in[0,S_k/D_f]$ the MISE
    of $F_{nh}$ admits the exact expression ${\rm MISE}(h)=n^{-1}\psi(F)-n^{-1}\psi(K)h$. From
    this, it follows that for any fixed $h\in(0,S_k/D_f]$ the kernel estimator $F_{nh}$
    presents an asymptotic first-order reduction in MISE over the empirical estimator; that is,
    its MISE is of order $n^{-1}$ as for the empirical estimator, yet with a strictly smaller
    constant (namely, $\psi(F)-\psi(K)h<\psi(F)$). As a result, over the class of distributions
    with $D_f$ bounded by a constant (say, $D_f\leq M$) the kernel estimator with bandwidth $h=S_k/M$ is
    strictly more efficient than the empirical distribution function $F_n$. This is in contrast with the more common
    case (i.e., $S_k=0$ or $D_f=\infty$) where it is well-known that the asymptotic improvement
    of $F_{nh}$ over $F_n$ is only of second order, in the sense that ${\rm MISE}(h_{0n})$
    admits the asymptotic representation $n^{-1}\psi(F)-cn^{-p}+o(n^{-p})$ for some $p>1$ and
    $c>0$ (see, e.g., \citealp{Jon90}, and \citealp{SX97}).
\end{remark}

\subsection{Sinc kernel distribution function estimator}

In this section we consider the sinc kernel, defined by ${\rm sinc}(x)=\sin(x)/(\pi x)$ for $x\neq0$ and ${\rm sinc}(0)=1/\pi$. This function is not integrable, so the sinc kernel density estimator inherits this undesirable property, but such a defect can be corrected as described in \cite{GHU03}. Nevertheless, the sinc kernel is square integrable, and as such the sinc kernel density estimator achieves certain optimality properties with respect to the MISE \citep{Dav77}, that make the sinc kernel useful for density estimation (see \citealp{GHU07}, or \citealp{Tsy09}, Section 1.3).

\citet[][Section 3]{Abd93} provided a careful study showing that it also makes sense to use the MISE criterion for kernel distribution function estimators based on the integrated sinc kernel. However, it is not so clear from his developments how the sinc kernel distribution function estimator is explicitly defined, nor the asymptotic properties of the optimal bandwidth sequence in this case, since Theorem \ref{thm:limit} above can not be directly applied given that the sinc kernel is not integrable. This section contains a detailed treatment of these issues.

First, note that the definition of the integrated kernel $K(x)=\int_{-\infty}^x{\rm sinc}(z)dz$ has to be understood in the sense of Cauchy principal value, i.e. $K(x)=\lim_{M\to\infty}\int_{-M}^x{\rm sinc}(z)dz$, because the integral is not Lebesgue-convergent. A simpler way to express such principal value is $K(x)=\frac12+{\rm Si}(x)$, where ${\rm Si}(x)=\int_0^x{\rm sinc}(z)dz$ is the sine integral function (with the usual convention that $\int_a^b=-\int_b^a$ if $b<a$). This yields the following explicit form for the sinc kernel distribution function estimator:
\begin{equation}\label{sinckde}
F_{nh}^{\rm sinc}(x)=\tfrac12+n^{-1}\sum_{j=1}^n{\rm Si}\big(h^{-1}(x-X_j)\big).
\end{equation}

An alternative, and perhaps more natural, derivation of (\ref{sinckde}) is found through the use of inversion formulas. The sinc kernel density estimator with bandwidth $h=1/T$ is readily obtained from the inversion formula $f(x)=(2\pi)^{-1}\int e^{-itx}\varphi_f(t)dt$ by replacing $\varphi_f$ with the empirical characteristic function $\varphi_n(t)=n^{-1}\sum_{j=1}^ne^{itX_j}$, conveniently truncated to get a finite integral $(2\pi)^{-1}\int_{-T}^Te^{-itx}\varphi_n(t)dt$ \citep[see for instance][p. 774]{Ch92}. An inversion formula relating $F$ and $\varphi_f$ is the so-called Gil-Pelaez formula $F(x)=\frac12-\frac1\pi\int_0^\infty t^{-1}\Im\{e^{-itx}\varphi_f(t)\}dt$, with $\Im\{z\}$ standing for the imaginary part of a complex number $z$, which is valid for a continuous $F$ in the principal value sense \citep{Gur48}. Reasoning as before, replacing $\varphi_f$ with $\varphi_n$ and restricting the domain of integration to $[0,1/h]$, results in the same sinc kernel distribution function estimator shown in (\ref{sinckde}).

As a square integrable function, the Fourier transform of the sinc kernel is the indicator function
of the interval $[-1,1]$. In this sense, \cite{Abd93} showed that the IV and ISB formulas of
Theorem \ref{thm:Fourier} above remain valid for the sinc kernel distribution estimator, as long as
the square integrability of $f$ is added to its assumptions, leading to the following simple exact
MISE formula for $h>0$:
\begin{equation}\label{eq:MISEsinc}
{\rm MISE}(h)=(n\pi)^{-1}\int_{0}^{1/h} t^{-2} \{1-|\varphi_f(t)|^2\}dt + \pi^{-1}\int_{1/h}^{\infty} t^{-2} |\varphi_f(t)|^2dt.
\end{equation}
Straightforward differentiation shows that the critical points of such a MISE function are located at any value $h_\diamond$ such that $|\varphi_f(1/h_\diamond)|^2=(n+1)^{-1}$. This does not reveal, however, if such critical points are local minima or maxima. The following result shows the existence of a global minimizer $h_{0n}$ of (\ref{eq:MISEsinc}), and that Theorem \ref{thm:limit} above remains valid for the sinc kernel estimator. Note, again, that Theorem 2 of \cite{Ten06} on the existence of $h_{0n}$ can not be directly applied here because it relies on the assumption that the kernel function is integrable.

\begin{theorem}\label{thm:hsinc}
Assume that $f$ is square integrable and $\int|x|dF(x)<\infty$. Then, there exists a bandwidth
$h_{0n}$ that minimizes the MISE of the sinc kernel distribution function estimator. Moreover, if
$C_f=D_f$ then $h_{0n}\to 1/D_f$ as $n\to\infty$ and also $h_*=1/D_f$.
\end{theorem}

If it were integrable, the sinc kernel could be considered as a superkernel with $S_{\rm sinc}=T_{\rm sinc}=1$, so from Theorem \ref{thm:hsinc} it follows that all the remarks above about the limit behavior of $h_{0n}$ and the optimal MISE for superkernel distribution function estimators can be equally applied to the sinc kernel distribution estimator.

\section{Numerical examples}\label{sec:examples}

In this section we present some examples to further illustrate the usefulness and consequences of
Theorems \ref{thm:Fourier}, \ref{thm:limit} and \ref{thm:hsinc} above.

\subsection{Example 1}\label{ex:dlVP}

In this example we consider the so-called Jackson-de la Vall\'e Poussin distribution $F$, with density function
$$f(x)=\dfrac{3}{4 \pi} \left(\dfrac{\sin(x/2)}{x/2}\right)^4=\frac{9+3\cos(2x)-12\cos(x)}{2\pi x^4}$$
and whose characteristic function is shown in \citet[][p. 516]{But71} to be
$$
\varphi_f(t)=
\begin{cases}
\begin{array}{ll}
1-3t^2/2+3|t|^3/4, & \abs{t}\leq 1\\
(2-|t|)^3/4, &  1\leq \abs{t} \leq 2\\
0, & \abs{t}\geq 2
\end{array},
\end{cases}
$$
which implies that $C_f=D_f=2$.

%

As shown in Theorems \ref{thm:limit} and \ref{thm:hsinc}, since $C_f=D_f<\infty$ this distribution
(or any of its rescalings $F_a(x)=F(x/a)$ with $a>0$) represents a case where superkernel
distribution function estimators are asymptotically more efficient than the empirical distribution
function. To illustrate this fact, we include here a numerical comparison using two different
superkernels: the sinc kernel and a proper superkernel, the trapezoidal superkernel given by
$k(x)=(\pi x^2)^{-1}\{\cos x-\cos(2x)\}$, for which $S_k=T_k=1$ \citep[see][]{CMN07}.

%
%
%
%
%
%
%
%
%
%

It is not hard from Theorem \ref{thm:Fourier} (for the trapezoidal kernel) and (\ref{eq:MISEsinc})
(for the sinc kernel) to come up with an explicit formula for the exact MISE function in each case.
These exact MISE calculations allow to numerically compute the optimal bandwidth sequences $h_{0n}$
and the minimum MISE values. The optimal bandwidth sequences for both superkernel estimators
are shown in Figure \ref{fig:2} (left) as a function of $\log_{10} n$, where it is already
noticeable that they both have limit $1/2$, as predicted from theory.

\begin{figure}[t]
  \centering
\includegraphics[width=0.48\textwidth]{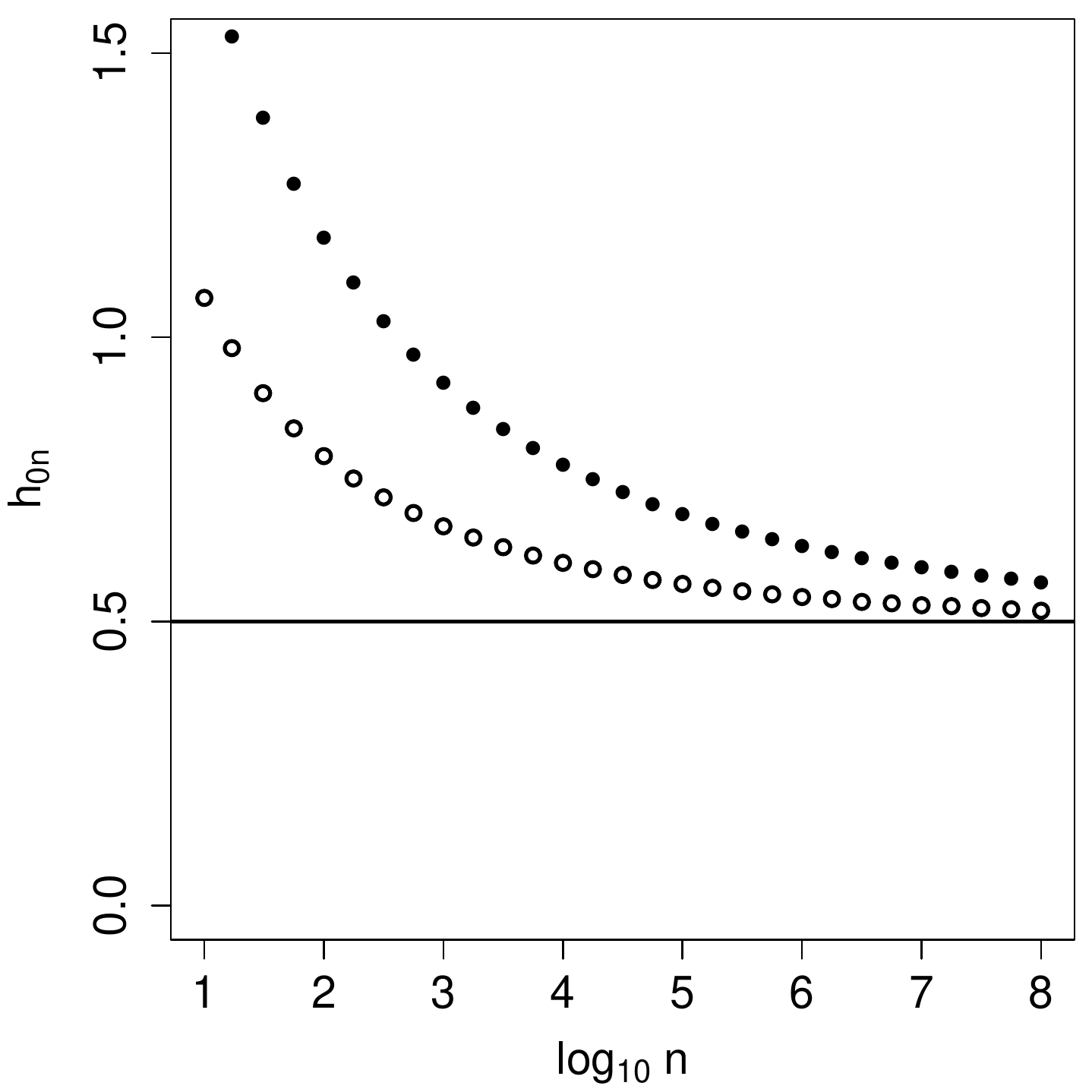}\hspace{0.04\textwidth}\includegraphics[width=0.48\textwidth]{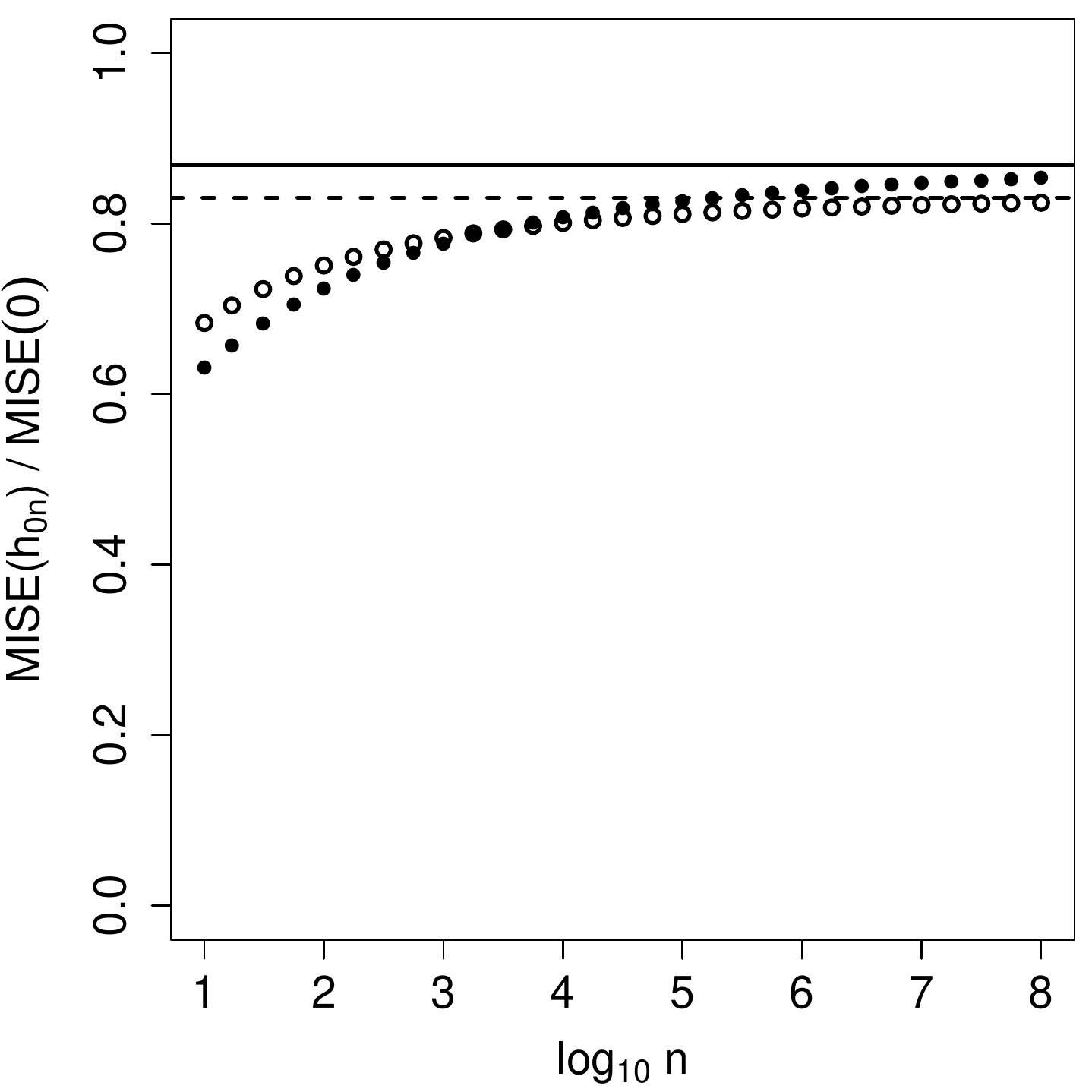}
  \caption{Optimal bandwidth sequence (left) and relative efficiency in MISE (right) for the estimation of the Jackson-de la Vall\'e Poussin distribution, as a function of $\log_{10}n$. The lines show the limit values. Solid circles and solid lines correspond to the trapezoidal superkernel and open circles and dashed lines correspond to the sinc kernel.}
  \label{fig:2}
\end{figure}

The right graph in Figure \ref{fig:2} shows the relative efficiency of both superkernel estimators
using optimal bandwidths with respect to the empirical estimator, namely ${\rm MISE}(h_{0n})/{\rm
MISE}(0)$, together with their asymptotic values, given by ${\rm MISE}(S_k/D_f)/{\rm
MISE}(0)=1-\psi(K)S_k/\{\psi(F)D_f\}$. Using (\ref{eq:F1F}) it follows that $\psi(F)=(96\log
2-43)/(8\pi)$, and $\psi(K)$ equals $(4\log 2-2)/\pi$ and $1/\pi$ for the trapezoidal and the sinc
kernel, respectively, resulting in asymptotic relative efficiencies of approximately 0.87 and 0.83
for the two superkernel estimators, as reflected on Figure \ref{fig:2}. For this distribution, the
trapezoidal kernel is more efficient than the sinc kernel up to about sample size $n=3000$, but
asymptotically the sinc kernel is slightly more efficient. Both are markedly more efficient than
the empirical distribution as was to be expected from asymptotic theory; besides, the gains are
even more substantial for low and moderate sample sizes.

\subsection{Example 2}\label{ex:normales}

In this second example we make use of the MISE expressions in terms of characteristic functions to
obtain exact MISE formulas for the case-study in which $F$ corresponds to the $N(0,\sigma^2)$
distribution and the integrated kernel is either the standard normal distribution function $\Phi$,
or the integrated sinc kernel, and we compare both estimators.

For this specific example the exact MISE formula for the density estimation problem was provided in
\cite{Fry76} making use the convolution properties of the normal density function, which are also
useful for deriving many other integral results for the normal density and its derivatives
\citep[see][]{AMPW95}.

However, convolution techniques seem to be of little use to find exact MISE expressions for kernel distribution function estimators in the normal case, where not even the estimation goal $F$ has an explicit formula. For this problem, it is convenient to work with exact expressions in terms of characteristic functions. For instance, using (\ref{eq:F1F}) it immediately follows that the MISE for the empirical distribution function equals $n^{-1}\pi^{-1/2}\sigma$ and, similarly, it is not hard to show that for the kernel estimator with the normal kernel
$$\pi^{1/2}{\rm MISE}(h)=n^{-1}\{(h^2+\sigma^2)^{1/2}-h\}+\{(2h^2+4\sigma^2)^{1/2}-(h^2+\sigma^2)^{1/2}-\sigma\}$$
and with the sinc kernel
$$\pi{\rm MISE}(h)=(1+n^{-1})\big\{he^{-\sigma^2/h^2}+2\sigma\sqrt{\pi}\Phi\big(\sigma\sqrt{2}\big/h\big)\big\}-n^{-1}h-(2+n^{-1})\sigma\sqrt{\pi}.$$
In Figure \ref{fig:3} we show the relative efficiency ${\rm MISE}(h_{0n})/{\rm MISE}(0)$ as a function of $\log_{10}n$ for $\sigma=1$ for both kernel estimators with respect to the empirical distribution function. Here, all the three estimators are asymptotically equally efficient, in the sense that the relative efficiency converges to 1 as $n\to\infty$. However, it is clear that this convergence is much slower for the sinc kernel estimator, which is more efficient that the normal kernel estimator for sample sizes as low as $n=50$.
\begin{figure}[t]
  \centering
\includegraphics[width=0.48\textwidth]{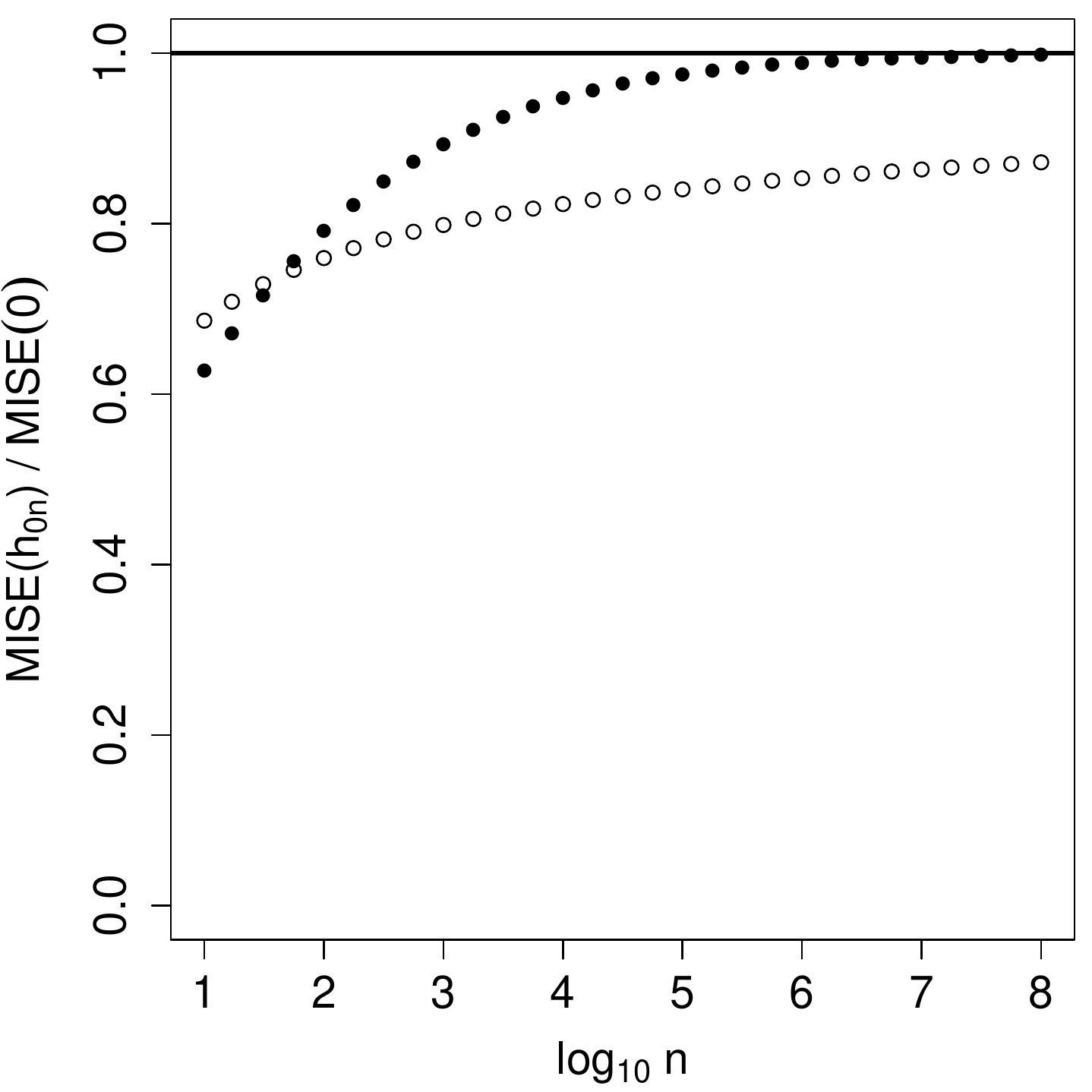}
  \caption{Relative efficiency in MISE for the estimation of standard normal distribution, as a function of $\log_{10}n$. The line shows the limit value. Solid circles correspond to the normal kernel and open circles correspond to the sinc kernel. }
  \label{fig:3}
\end{figure}

%
%
%
%

\section{Proofs}\label{sec:proofs}

For $h>0$, the statement of Theorem \ref{thm:Fourier} is contained within the proof of Proposition
2 in \cite{Abd93}. Therefore, it only remains to show the case $h=0$; i.e., Equation
(\ref{eq:F1F}). This formula is valid in the more general situation where $F$ is not necessarily a distribution
function, but an integrated kernel with finite first order moment, as shown in the following lemma.

\begin{lemma}\label{lem:intF1F}
Suppose that $K(x)=\int_{-\infty}^xk(y)dy$, where $k$ is an integrable function with $\int
k(y)dy=1$ and $\int|yk(y)|dy<\infty$. Then,
$$\int K(x)\{1-K(x)\}dx=(2\pi)^{-1}\int t^{-2}\{1-|\varphi_k(t)|^2\}dt.$$
\end{lemma}

\begin{proof}
It is not hard to show that $K(x)\{1-K(x)\}=\int\{I_{[y,\infty)}(x)-K(x)\}^2k(y)dy$, where $I_A$
stands for the indicator function of a set $A$. Moreover, reasoning as in the proof of Proposition
2 in \cite{Abd93}, it follows that the condition $\int|yk(y)|dy<\infty$ guarantees that
$\int|I_{[y,\infty)}(x)-K(x)|dx<\infty$ for all $y$, which implies that the function
$G_y(x)=I_{[y,\infty)}(x)-K(x)$ is square integrable, since $K$ is bounded (because
$|K(x)|\leq\int|k(y)|dy$ for all $x$). Therefore, by Parseval's identity,
$\int\{I_{[y,\infty)}(x)-K(x)\}^2dx=(2\pi)^{-1}\int|\varphi_{G_y}(t)|^2dt$. The Fourier transform
of $G_y$ is shown to be $(-it)^{-1}\{e^{ity}-\varphi_k(t)\}$, since splitting the integration
region and using integration by parts,
\begin{align*}
-it\varphi_{G_y}(t)&=it\int_{-\infty}^y e^{itx}K(x)dx-it\int_y^{\infty} e^{itx}\{1-K(x)\}dx\\
&=K(y)e^{ity}-\int_{-\infty}^y e^{itx}k(x)dx+\{1-K(y)\}e^{ity}-\int_y^{\infty} e^{itx}k(x)dx\\
&=e^{ity}-\varphi_k(t).
\end{align*}
Thus, $\int\{I_{[y,\infty)}(x)-K(x)\}^2dx=(2\pi)^{-1}\int
t^{-2}\big[1+|\varphi_k(t)|^2-2\Re\big\{e^{-ity}\varphi_k(t)\big\}\big]dt$, where $\Re\{z\}$
denotes the real part of a complex number $z$. This finally leads to
\begin{align*}
\int K(x)\{1-K(x)\}dx&=\iint \{I_{[y,\infty)}(x)-K(x)\}^2k(y)dxdy\\
&=(2\pi)^{-1}\iint t^{-2}\big[1+|\varphi_k(t)|^2-2\Re\big\{e^{-ity}\varphi_k(t)\big\}\big]k(y)dydt\\
&=(2\pi)^{-1}\int t^{-2}\{1-|\varphi_k(t)|^2\}dt,
\end{align*}
where the last line follows from the fact that $\int
e^{-ity}\varphi_k(t)k(y)dy=\varphi_k(-t)\varphi_k(t)=|\varphi_k(t)|^2$.
\end{proof}

The proof of Theorem \ref{thm:limit} is immediate from the following lemma.

\begin{lemma}\label{lemaSK}
Assume that $F$ and $K$ satisfy the assumptions of Theorem \ref{thm:limit}. Then,
$$S_k/D_f \leq \inf_{n \in \mathbb{N}} h_{0n}\leq\limsup_{n \to \infty} h_{0n} \leq h_* \leq \min \lbrace S_k/C_f, T_k/D_f \rbrace.$$
\end{lemma}

\begin{proof}
First notice that ${\rm ISB}(h)=0$ for all $h \in [0,S_k/D_f]$, since using Theorem
\ref{thm:Fourier}
\begin{align*}
0&\leq  \pi\,{\rm ISB}(h)=\int_0^{\infty} t^{-2}|1-\varphi_k(th)|^2|\varphi_f(t)|^2dt\\
 &\leq \int_0^{S_k/h} t^{-2}|1-\varphi_k(th)|^2|\varphi_f(t)|^2dt + \int_{D_f}^{\infty} t^{-2}|1-\varphi_k(th)|^2|\varphi_f(t)|^2dt=0,
\end{align*}
with the last equality due to the facts that $\varphi_k(th)=1$ for $t \in [0,S_k/h]$ and
$\varphi_f(t)=0$ for $t \geq D_f$ by definition of $S_k$ and $D_f$, respectively.

Therefore, for $h\in[0,S_k/D_f]$ the MISE reduces to the IV, and admits the exact expression ${\rm
MISE}(h)=n^{-1}\{\psi(F)-\psi(K)h\}$ because, again using Theorem \ref{thm:Fourier}, noting the
expression (\ref{eq:F1F}) for $\psi(F)$, taking into account the definition of $S_k$ and $D_f$ and
making the change of variable $s=th$, we obtain
\begin{align*}
{\rm IV}(h)&=(n\pi)^{-1}\int_0^\infty t^{-2}|\varphi_k(th)|^2\{1-|\varphi_f(t)|^2\}dt\\
&=(n\pi)^{-1}\int_0^\infty t^{-2}\{1-|\varphi_f(t)|^2\}dt-(n\pi)^{-1}\int_0^\infty t^{-2}\{1-|\varphi_k(th)|^2\}\{1-|\varphi_f(t)|^2\}dt\\
&=n^{-1}\psi(F)-(n\pi)^{-1}\int_{S_k/h}^\infty t^{-2}\{1-|\varphi_k(th)|^2\}dt=n^{-1}\psi(F)-n^{-1}\psi(K)h.
\end{align*}
Since the MISE function is linear in $h$ with negative slope in $[0,S_k/D_f]$, its minimum has to
be attached at some point greater than $S_k/D_f$, hence we obtain the first inequality.

On the other hand, reasoning as in \cite{CMNP07} it is possible to show that ${\rm ISB}(h)>0$ for
$h>S_k/C_f$ and for $h>T_k/D_f$, thus yielding the last inequality. Finally, denote $h_L=\limsup_{n
\rightarrow \infty}h_{0n}$ and assume that $h_L>h_*$. Then, the continuity of ${\rm ISB}(h)$ with
respect to $h$ \citep[][Proposition 1]{Ten06} entails that there is a subsequence $h_{0n_k}$ such
that, as $k\to\infty$, ${\rm ISB}(h_{0n_k})\to{\rm ISB}(h_L)$ with ${\rm ISB}(h_L)>0$ since we
are assuming $h_L>h_*$. But from (\ref{IV}) and (\ref{ISB}) it immediately follows that, for every
fixed $h$, ${\rm MISE}_{n_k}(h)\to{\rm ISB}(h)$ as $k\to\infty$, so we obtain that the following chain
of inequalities
$${\rm ISB}(h)=\lim_{k\to\infty}{\rm MISE}_{n_k}(h)\geq\lim_{k\to\infty}{\rm MISE}_{n_k}(h_{0n_k})\geq\lim_{k\to\infty}{\rm ISB}(h_{0n_k})={\rm ISB}(h_L)>0$$
is valid for every fixed $h$, implying that $\lim_{h\to0}{\rm ISB}(h)\geq{\rm ISB}(h_L)>0$, which contradicts Proposition 1 in \cite{Ten06}, where it is shown that ${\rm ISB}(h)\to0$ as $h\to0$. Hence, it should be $h_L\leq h_*$, as desired.
\end{proof}

Finally, we show the proof of Theorem \ref{thm:hsinc}. We focus only on the statement about the existence of the optimal bandwidth sequence, since the arguments showing the limit behavior can be adapted from the proof of Lemma \ref{lemaSK} above.

\begin{proof}[Proof of Theorem \ref{thm:hsinc}]
It is clear from (\ref{eq:MISEsinc}) and (\ref{eq:F1F}) that $\lim_{h\to0}{\rm MISE}(h)=n^{-1}\psi(F)$. Moreover, $\lim_{h\to\infty}{\rm MISE}(h)=\infty$, since $\varphi_f(0)=1$ and by continuity it is possible to take $\delta>0$ such that $|\varphi_f(t)|^2>\frac12$ for all $0\leq t\leq\delta$, so this yields $\int_0^\infty t^{-2}|\varphi_f(t)|^2dt\geq\frac12\int_0^\delta t^{-2}dt=\infty$. These two limit conditions, together with the fact that ${\rm MISE}(h)$ is a continuous function, imply that the existence of a minimizer of the MISE is guaranteed if there is some $h_1>0$ such that ${\rm MISE}(h_1)<n^{-1}\psi(F)$. But from (\ref{eq:MISEsinc}) we have
$${\rm MISE}(h)-n^{-1}\psi(F)=-(n\pi)^{-1}h+(1+n^{-1})\pi^{-1}\int_{1/h}^\infty t^{-2}|\varphi_f(t)|^2dt$$
so that using the Riemann-Lebesgue lemma and the dominated convergence theorem, it follows that
$$\lim_{h\to0}h^{-1}\{{\rm MISE}(h)-n^{-1}\psi(F)\}=-(n\pi)^{-1}<0,$$ which entails that there is
some $h_1>0$ fulfilling the aforementioned desired property.
\end{proof}

\bigskip

\noindent{\bf Acknowledgments.}  This work has been supported by grant MTM2010-16660 from the
Spanish Ministerio de Ciencia e Innovaci\'on and by the Centro de Matem\'atica da Universidade de
Coimbra (funded by the European Regional Development Fund through the program COMPETE and by the
Portuguese Government through the FCT - Funda\c{c}\~ao para a Ci\^encia e Tecnologia under the
project PEst-C/MAT/UI0324/2011).

\bibliographystyle{apalike}

\end{document}